\documentclass[12pt]{amsart}
\usepackage{a4,amsmath,amssymb,epic}
\usepackage{graphicx}
\usepackage{multicol}
\usepackage{rotating}
\usepackage[margin=3.4cm]{geometry}
\usepackage{amsthm}
\usepackage{enumitem}
\usepackage[all]{xy}
\usepackage{tikz}

\allowdisplaybreaks
\numberwithin{equation}{subsection}
\theoremstyle{joltthm}

\newtheorem{lem}{Lemma}

\newtheorem{thm}[lem]{Theorem}
\newtheorem{cor}[lem]{Corollary}

\newtheorem{defn}[lem]{Definition}

\newtheorem{rmk}{Remark}

\theoremstyle{joltthm}

\newcommand{\Hom}{\operatorname{Hom}}
\newcommand{\Head}{\operatorname{head}}
\newcommand{\rad}{\operatorname{rad}}
\newcommand{\Soc}{\operatorname{soc}}
\newcommand{\Top}{\operatorname{head}}
\newcommand{\End}{\operatorname{End}}
\newcommand{\ind}{\operatorname{ind}}
\renewcommand{\le}{\leqslant}

\newcommand{\SL}{\mathrm{SL}}

\newcommand{\T}{\mathrm{T}}

\renewcommand{\L}{\mathrm{L}}

\newcommand{\im}{\operatorname{im}}

\newcommand{\soc}{\operatorname{soc}}

\begin{document}

\title[A reciprocity result]{A Reciprocity Result for Projective Indecomposable Modules 
of Cellular Algebras and BGG Algebras}

\author{C.~Bowman and S.~Martin}

\address{%
C.~Bowman \\ 
Corpus Christi College \\ 
Cambridge   \\
C.Bowman@dpmms.cam.ac.uk							}  

\address{%
S. ~Martin \\ 
Magdalene College \\ 
Cambridge    \\
s.martin@dpmms.cam.ac.uk}

\date{30th January 2012}

\maketitle

\begin{abstract}  We show that an adaptation of Landrock's 
Lemma for symmetric algebras also holds for cellular 
algebras and BGG algebras.  This is a result relating 
the radical layers of any two projective modules. 
As a corollary we deduce that BGG reciprocity 
respects Loewy structure.
 \end{abstract}

\section{Introduction}

The Loewy structure of finite groups received considerable attention in the late seventies and eighties.  The calculation of the Loewy structure of particular finite group algebras over fields of characteristic $p$ was carried out by Alperin, Benson, and Koshitani, amongst others.  This detailed information for a given group is difficult to obtain.  The main tools used by these authors were those of local representation theory 
 and Landrock's Lemma. 

 Landrock's Lemma is a reciprocity result relating the Loewy layers of projective indecomposable modules of a symmetric algebra; this reciprocity significantly reduces the work required in the calculation of Loewy structure.

The central result of this paper is an adaptation of Landrock's Lemma (see \cite{al2} Lemma 1.9.10) for symmetric algebras to the setting of cellular and BGG algebras. 

 In the case of cellular algebras this leads to a reduction in the calculation of Loewy structure comparable to the case of symmetric algebras; indeed in the case of group algebras of symmetric groups (which are both symmetric and cellular algebras) we are simply restating the original lemma.
 
   In the case of BGG algebras we obtain as a corollary that BGG reciprocity (formerly seen as a numerical, character theoretic statement) is the shadow of a stronger statement relating the radical filtrations of projective and standard modules.  This not only gives a better understanding of the BGG reciprocity principle, but drastically reduces the computation of the Loewy structure of projective modules to that of standard modules.

We belatedly define rigidity and fix some notation.  The socle of an $A$-module $M$ is defined to be the largest semisimple submodule of $M$, and denoted $\Soc M$.  Now consider the socle of $M/\soc M$, and let $\soc^2 M$ denote the submodule of $M$ containing $\soc M$ such that $\soc^2 M /\soc M$ is the socle of $M/\soc M$. We then inductively define the socle series $\soc^s M$.  The socle layers are the semisimple quotients $\soc_s M=\soc^{s}M/\soc^{s-1}M$.
The radical of of an $A$-module $M$ is defined to be the smallest module with a semisimple quotient, and denote $\rad M$.  We then let $\rad^2 M = \rad (\rad M)$ and inductively define the radical series, $\rad^s M$, of $M$.  The radical layers are then the semisimple quotients $\rad_s M = \rad^{s}M/\rad^{s-1}M$.

We let $\Top M$ denote the first radical layer.  The length of the radical series coincides with the length of the socle series.  For a given module, $M$, this common number is called the Loewy length denoted $ll(M)$.  An $A$-module is \emph{rigid} if the radical and socle series coincide, i.e. if $\rad_sM=\soc_{ll(M)-s}M$.

\section{Cellular Algebras and BGG Algebras}
\subsection{Cellular algebras} 
We recall the original definition of a cellular algebra.
\begin{defn} \rm 
An associative $k$-algebra $A$ is called a cellular algebra with cell datum $(\pi;M;C; i)$ if the following conditions are satisfied:

(C1) The finite set $\pi$ is partially ordered.  Associated with each $\lambda \in \pi$ there is a finite set $M(\lambda)$.  The algebra $A$ has $k$-basis $C^\lambda _{S,T}$ where $(S,T)$ runs through all elements of $M(\lambda)$ for all $\lambda \in \pi$.

(C2) The map $i$ is a $k$-linear anti-automorphism of $A$ with $i^2 = \text{id}$ which sends
each $C^\lambda_{S,T}$ to $C^\lambda_{T,S}$.

(C3) For each $\lambda \in \pi$ and $S,T \in M(\lambda)$ and each $a \in A$ the product $aC^\lambda_{S,T}$ can be written as $(\sum_{U \in M(\lambda)}r_a(U,S)C^\lambda_{U,T}) + r'$ where $r'$ is a linear combination of basis elements with upper index strictly less than $\lambda$, and where coefficients $r_a(U, S) \in k$ do not depend on $T$.
\end{defn}
The following definition has been shown to be equivalent (see \cite{xi}).
\begin{defn} \rm
Let $A$ be a $k$-algebra.  Assume there is an anti-automomorphism $i$ on $A$ with $i^2= \text{id}$.  A two-sided ideal $I$ in $A$ is called a \emph{cell ideal} if and only if $i(I)=I$ and there exists a left ideal $L \subset I$ such that $L$ has finite $k$-dimension and that there is an isomorphism of $A$-bimodules $\alpha:I \cong L \otimes_k i(L)$ making the 
 diagram 
 below
 commutative:
\[\def\objectstyle{\scriptstyle} \xymatrix@=5pt{
&I \ar@{>}[rrrr]^{\alpha}\ar@{>}[dddd]^{i}	&&&&L \otimes_k i(L)\ar@{>}[dddd]^{x \otimes y \mapsto i(y) \otimes i(x)}\\
&\\
&\\
&\\
&I \ar@{>}[rrrr]^{\alpha}			&&&&L \otimes_k i(L)	
}\]

The algebra $A$ is called \emph{cellular} if and only if there is a vector space decomposition $A=I'_1 \oplus \ldots \oplus I'_n$ with $i(I'_j)=I'_j$ for each $j$ and such that setting $I_j=\oplus_{k=1}^jI'_l$ gives a chain of two-sided ideals of $A$ and for each $j$ the quotient $I'_j=I_j/I_{j-1}$ is a cell ideal of $A/I_{j-1}$.
\end{defn}

\subsection{Quasi-hereditary algebras}\label{projectivereciprocity}  There are several equivalent definitions of a quasi-hereditary algebra, we use the definition from \cite{dual} as it allows us to fix convenient notation in the process.  

Let $A$ be an associative finite dimensional $k$-algebra, for $k$ an algebraically closed field.  Let $L(\lambda)$ constitute a full set of simple $A$-modules as $\lambda$ varies over an indexing poset $\pi$.  Let $P(\lambda)$ denote the projective cover and let $I(\lambda)$ denote the injective hull of $L(\lambda)$.  By $\Delta(\lambda)$ we denote the maximal factor module of $P(\lambda)$ such that all composition factors are of the form $L(\mu)$ for $\mu \leq \lambda$ in $\pi$.  The modules $\Delta(\lambda)$ are called standard modules.  The costandard modules, $\nabla(\lambda)$, are defined in a dual manner.

\begin{defn} \rm
Let $A$ be an algebra with standard modules $\Delta$.  The algebra is called quasi-hereditary if
\begin{enumerate}[leftmargin=*,itemsep=0.1em]
\item $\End_A(\Delta(\lambda))\cong k $ for all $\lambda \in \pi$,
\item every projective module has a filtration by standard modules.
\end{enumerate}
\end{defn}

\begin{defn} \rm
A duality on $A$-mod is a contravariant, exact, additive functor $\delta$ from $A$-mod to itself such that $\delta \circ \delta$ is naturally equivalent to the identity functor on $A$-mod and $\delta$ induces a $k$-linear map on the vector spaces $\Hom_A(M,N)$ for all $M, N \in A$-mod.
\end{defn}

\begin{defn} \rm
Let $A$ be a quasi-hereditary algebra.  If there exists a duality functor $\delta$ on $A$-mod such that $\delta(L(\lambda))=L(\lambda)$ for all $\lambda$, then $A$ is said to be a BGG-algebra.
\end{defn}

BGG algebras were defined in \cite{CPSdual} in order to generalise the famous BGG-reciprocity principle for category $\mathcal{O}$ to the general setting of highest weight categories.  They show that for $\lambda, \mu \in \pi$ we have that $[P(\mu):\Delta(\lambda)]=[\Delta(\lambda):L(\mu)]$, and so calculating the characters of projective modules is equivalent to calculating the characters of the smaller standard modules.

\subsection{The involution}  We have that the anti-involution $i$ defines a duality functor, $\sharp$, on a cellular algebra in the following way.  Let $V$ be a left $A$-module, then $V^\sharp=\Hom(V,k)$ is a left $A$-module with the action:
\begin{align*}
(a \circ f)(v)=(f(i(a)v)	 \ \ \ (f \in V^*, a \in A, v \in V).
\end{align*}

We extend $\sharp$ to a duality on BGG algebras by letting $\sharp$ act as $\delta$ on a BGG algebra.
We have for BGG algebras (see \cite{dual}) and cellular algebras (see \cite{Cao}) that the duality $\sharp$ fixes simple modules, and interchanges their projective and injective hulls.

\section{The Result}

The following proposition is an adaptation of Landrock's Lemma (see \cite{al2} Lemma 1.9.10) for symmetric algebras to the setting of cellular algebras and BGG algebras.  

\begin{thm}\label{thm}
Let $A$ be a cellular algebra or a BGG algebra.  For $\lambda, \mu \in \pi$ we have the following reciprocity:
\begin{equation*}\tag{$\dagger$}
 [ \rad_sP(\mu): L(\lambda)] =[\rad_sP(\lambda):L(\mu)]. \\
\end{equation*}
\end{thm}

\begin{cor}
Let $A$ denote a BGG algebra with poset $\pi$.  For weights $\lambda, \mu \in \pi$ we have that
\begin{align*}
[ \rad_sP(\mu):\Head\Delta(\lambda)]&=[ \rad_s\Delta(\lambda):L(\mu)],\tag{$\dagger\dagger$}
\end{align*}
and so BGG reciprocity respects Loewy structure. 
\end{cor}

\begin{rmk} \rm
By $[ \rad_sP(\mu):\Top\Delta(\lambda)]$ we mean the number of successive subquotients $\Delta(\lambda_j)$ in a \emph{fixed} $\Delta$-filtration of $P(\mu)$, such that $\lambda_j = \lambda$ and there is a surjection $\rad^s P(\mu)\to \Delta(\lambda) $ which carries the subquotient $\Delta(\lambda_j)$ onto $\Delta(\lambda)$. 

Heuristically, we think of this as follows. The $\dagger$ reciprocity relates the occurrences of simple modules in projective modules.  The $\dagger \dagger$ reciprocity relates how the Weyl filtrations sit within the Loewy structure of the projective modules --- in other words, by $\Head\Delta(\lambda)$ we mean the head of the Weyl module, not a simple module isomorphic to the head.  
\end{rmk}

\begin{rmk} \rm
We may replace projective by injective, standard by costandard, and radical by socle to obtain the dual of the above theorem and corollary.
\end{rmk}

\begin{cor}
Let $A$ denote a cellular algebra or a BGG algebra.  Let $\bar{A}=A / \rad^sA$, $\bar{P}(\lambda)=P(\lambda)/\rad^sP(\lambda)$.  Then $\bar{A}$ is a finite dimensional algebra with a complete list of simple modules given by $L(\lambda)$ as $\lambda$ ranges over $\pi$, each with corresponding prinicipal indecomposable modules $\bar{P}(\lambda)$.  Denote the Cartan matrix by $\bar{C}_{ij}$.  Then the Cartan matrix $\bar{C}_{ij}$ is symmetric.
\end{cor}

\begin{lem}\label{land}
Let $\beta_1, \ldots, \beta_{m_s}$ be a basis of a complement to $\Hom_{A}(P(\nu)/\rad^{s-1}$ $P(\nu),I(\mu))$ in $\Hom_{A}(P(\nu)/\rad^{s}P(\nu),I(\mu))$.  Then let 
\begin{align*}
V=\sum_{i=1}^{m_s}\im (\beta_i) \subseteq I(\mu), \ \ \ \ \ W / \rad^{s-1}P(\nu)=\bigcap_{i=1}^{m_s} \ker (\beta_i).
\end{align*}
Then we have that 
\begin{align*}
m_s= \dim_k(\Hom_{A} (L(\mu),P(\nu) / W)) &= \dim_k (\Hom_{A}(V, L(\nu))), \\&=\dim_k (\Hom_{A}((L(\nu), V^\sharp)).
\end{align*}
\end{lem}

\begin{proof} This follows as any $\beta=\sum b_i \beta_i$ is of Loewy length $s$ and if $\eta$ is the canonical homomorphism $V \to V/\rad V$ then $\beta_1\eta, \ldots , \beta_{m_s}\eta$ are linearly independent and so $V/\rad V \cong L(\nu)^{m_s}$, which gives the first equality.  The second equality follows from properties of the duality.
\end{proof}

\noindent
{\bf Proof of Theorem 1.} \quad
We define:
\begin{align*}
a_\nu &=\dim(\Hom_{A}(L(\nu), \rad^{s-1}P(\mu)/\rad^sP(\mu))), \\
\text{and } a_\mu &=\dim(\Hom_{A}(L(\mu), \rad^{s-1}P(\nu)/\rad^sP(\nu))).  
\end{align*}
By symmetry we need only prove that $a_\mu \leq a_\nu$. We have that $\Soc (P/W)= L(\mu)^{a_\mu}$ because any simple module not isomorphic to $L(\mu)$, or from a previous radical layer, is in $\bigcap_{i=1}^{m_s} \ker (\beta_i)$, and so by Lemma \ref{land} we have that:
\begin{align*}
a_\mu = m_s &=  \dim_k (\Hom_{A}(L(\mu),P(\nu) / W)),\\
 &= \dim_k (\Hom_{A}(V, L(\nu))), \\&=\dim_k (\Hom_{A}(L(\nu), V^\sharp)	). 
\end{align*}
Now, $V^\sharp$ has simple head isomorphic to $L(\mu)$ and is therefore a homomorphic image of $P(\mu)$.  Also $V^\sharp$ has Loewy length $s$ and $\rad^{s-1}V^\sharp=L(\nu)^{m_s}=L(\nu)^{a_\mu}$ and therefore we have that $a_\mu \leq a_\nu$.  This proves $\dagger$. \qed

\noindent{\bf Proof of Corollary 2.}\quad
In order to prove $\dagger \dagger$ we proceed by induction on the partial ordering of $A$.  If $\pi$ is a singleton, then the algebra is semisimple and we are done.  

If not, let $\lambda$ be a maximal element of $\pi$, then by induction the result holds for $A/\langle e_\lambda\rangle$ where $\langle e_\lambda\rangle$ is the ideal generated by the primitive central idempotent corresponding to the weight $\lambda$.  This quotient is then a BGG algebra with respect to the indexing set $\pi \backslash \{ \lambda\}$.
By our inductive assumption we only need show that the reciprocity holds for pairs $\mu\in \pi\backslash \{\lambda\}$ and $\lambda$.  We have that $\lambda$ is maximal and therefore by
Section 2.2
 we have that $P(\lambda)\cong\Delta(\lambda)$.  
 
 Now by applying $\dagger$ to $\mu \in \pi\backslash \{\lambda\}$ and $\lambda$ we have that:
\begin{align*}
[\rad_s\Delta(\lambda):\text{L}(\mu)] &= [\rad_sP(\lambda):\text{L}(\mu)],\\ &=[\rad_sP(\mu):L(\lambda)].
\end{align*}  Now as projective modules have a Weyl filtration, and there is no heavier weight than $\lambda$ in $\pi$, we have that $\dagger \dagger$ follows. \qed

\section{Applications to algebraic groups and their Schur algebras}

\subsection{Algebraic groups}
Let $G$ be a semisimple, simply connected linear algebraic group over
an algebraically closed field $k$ of positive characteristic $p$. We
fix a Borel subgroup $B$ and a maximal torus $T$ with $T \subset B
\subset G$ and we let $B$ determine the negative roots. We write $X =
X(T)$ for the character group of $T$, and $Y=Y(T)$ denote the co-character group.  We let $X^+$ denote the set of
dominant weights, and let $X_1$ denote the set of restricted weights, that is:
\[
  X_1 := \{ \nu \in X^+ \mid \langle \alpha^\vee, \nu \rangle \le
p \text{ for all simple roots } \alpha \}.
\]

By $G$-module we always mean a rational $G$-module,
i.e.\ a $K[G]$-comodule, where $K[G]$ is the coordinate algebra of
$G$.  For each $\lambda \in X^+$ we have the following (see
\cite{Jantzen}) finite dimensional $G$-modules:
\[
\begin{tabular}{ll}
$\L(\lambda)$ & simple module of highest weight $\lambda$;\\
$\Delta(\lambda)$ & Weyl module of highest weight $\lambda$;\\
$\nabla(\lambda)$ & $= \ind_B^G k_\lambda$; dual Weyl module of
 highest weight $\lambda$; \\
$\T(\lambda)$ & indecomposable tilting module of highest weight
$\lambda$
\end{tabular}
\]
where $k_\lambda$ is the 1-dimensional $B$-module upon which $T$ acts
by the character $\lambda$ with the unipotent radical of $B$ acting
trivially.  The simple modules $\L(\lambda)$ are contravariantly
self-dual.  The module $\nabla(\lambda)$ has simple socle isomorphic
to $\L(\lambda)$; the module $\Delta(\lambda)$ is isomorphic to the contravariant dual of $\nabla(\lambda)$, hence
has simple head isomorphic to $\L(\lambda)$.

\subsection{Schur Algebras}
Let $S(\pi)$ denote the generalised (quantum) Schur algebra corresponding to a finite saturated set $\pi$ of dominant weights of the semisimple, complex finite dimensional Lie algebra $\frak{g}$, in the sense of \cite{don1}.  Such an algebra is a BGG-algebra with standard modules $\Delta$.

\subsection{Projective tilting modules}\label{Jantzen} In  \cite{Zeit2} projective-injective modules of $(q)$-Schur algebras are studied and shown to be precisely the projective tilting modules for the algebra. 

Assume that the field $k$ is of characteristic $p\geq 2h-2$ .  By results of Jantzen in \cite{jan} we have, for $\lambda \in X_1$, that the $S(\pi)$-tilting module $T(2(p-1)\rho+\omega_0\lambda)$ is isomorphic to $P(\lambda)$ the projective module for the generalised Schur algebra $S(\pi)$ where $\pi=\{r \in X : r\leq 2(p-1)\rho\}$.  By the above, these tilting modules are both projective and injective, and therefore we can use both Corollary 2 and its dual to examine the radical and socle structure of these tilting modules, and hence answer questions concerning rigidity of these modules.

\subsection{Rigidity of tilting modules}
Following Andersen and Kaneda \cite{and} we consider the Loewy structure of tilting modules for algebraic groups.

\begin{cor}\label{tilt}
Let $G$ be a semisimple, simply connected algebraic group and let $p\geq 2h-2$.  
We have the following description of the radical layers of the tilting modules with highest weight from the region $(p-1)\rho+ X_1$, 
\begin{align*}
[ \rad_sT(2(p-1)\rho+\omega_0\mu):\Head \Delta(\lambda)]= [\rad_s\Delta(\lambda):L(\mu)], \\
[ \soc_sT(2(p-1)\rho+\omega_0\mu):\soc_1 \nabla(\lambda)]= [\soc_s\nabla(\lambda):L(\mu)].
\end{align*}
\end{cor}
\begin{proof}
We have that \ref{Jantzen} implies that as $S(\pi)$-modules, $T(2(p-1)\rho+\omega_0\lambda)\cong P(\lambda)\cong I(\lambda)$, substituting this into $\dagger\dagger$ our corollary follows.
\end{proof}

\begin{rmk} \rm
Therefore, determining the Loewy structure of the tilting modules from $(p-1)\rho+ X_1$ is equivalent to determining the Loewy structure of the Weyl modules from lower down in the linkage ordering.  This is illustrated in Section 4.5
.  In Corollary 4.4 of \cite{and} it is shown that these modules are rigid for $p\geq3h-3$.
\end{rmk}

\begin{rmk} \rm
Corollary \ref{tilt} follows from the fact that the tilting modules are both projective and injective module as modules for the Schur algebra.  The question of when tilting modules are projective is addressed in \cite{Zeit2} and a generalised Mullineux bijection is described which generalises Corollary \ref{tilt}.
\end{rmk}

\begin{rmk} \rm
The condition that $p\geq2h-2$ in Corollary \ref{tilt} is necessary.  A counter example of Corollary \ref{tilt} for small primes is given in Ringel's appendix to \cite{DM}.  This is because the tilting modules in Corollary \ref{tilt}  are not necessarily projective in characteristic $p<2h-2$.
\end{rmk}

\subsection{An example}\label{examp}
 Let $G=\text{SL}_3(k)$ for $k$ an algebraically closed field of characteristic $p \geq 2h-2=4$. 
  Take $T$ to be the subgroup of diagonal matrices and $B$ to be the group of lower triangular matrices.  We let $W$ denote the Weyl group, which is isomorphic to the symmetric group on three letters.  

 Let $X(T) \cong \mathbb{Z}^2$ denote the character group of $G$ and $Y \cong \text{Hom}_\mathbb{Z}(X,\mathbb{Z})$ denote the co-character group.  Setting $\alpha_1=(2,-1)$ and $\alpha_2=(-1,2)$ we have that $R=\{\alpha_1,\alpha_2,\alpha_1+\alpha_2\}$ are the roots of $G$ and that $S=\{\alpha_1,\alpha_2\}$ are the simple roots.  For $\gamma \in R$ we let $\check{\gamma} \in Y$ denote the dual root.

We enumerate the alcoves in the dominant chamber of $A_2$ as in Figure \ref{sl3label}.
The structure of the Weyl modules for $p\geq 5$ is well known (see for example \cite{DS} and \cite{ron}), and easily calculated by hyperalgebra calculations (which in this case are well within the range of GAP).  We have that:
\begin{align*}
\Delta(1)=[1], \ \ \ \Delta(2)=[2,1], \ \ \ \  \ \Delta(3)=[3,2], 
\end{align*}
\begin{align*}\Delta(4) =\begin{minipage}{34mm}
\def\objectstyle{\scriptstyle}
\xymatrix@=6pt{
  &4 \ar@{-}[dl] \ar@{-}[dr] & \\
3\ar@{-}[dr] &1 \ar@{-}[d] \ar@{-}[u] &3'\ar@{-}[dl] \\
  & 2&
} \; \end{minipage}, \ \ \ \ 
\Delta(5) =\begin{minipage}{34mm}
\def\objectstyle{\scriptstyle}
\xymatrix@=6pt{
&5\ar@{-}[dl] \ar@{-}[dr] \\
4\ar@{-}[drr] \ar@{-}[d] \ar@{-}[dr] &		&4' \ar@{-}[d] \ar@{-}[dl] \ar@{-}[dll] \\
3\ar@{-}[dr] &1 \ar@{-}[d] \ar@{-}[ul] &3'\ar@{-}[dl] \\
  & 2&
}  \end{minipage}.
\end{align*}

\begin{figure}[ht]
 \begin{center}
\begin{tikzpicture}[scale=1.5]
  \path (0,0) coordinate (origin);
  \path (60:1cm) coordinate (A1);
  \path (60:2cm) coordinate (A2);
  \path (60:3cm) coordinate (A3);
  \path (60:4cm) coordinate (A4);
  \path (120:1cm) coordinate (B1);
  \path (120:2cm) coordinate (B2);
  \path (120:3cm) coordinate (B3);
  \path (120:4cm) coordinate (B4);
  \path (A1) ++(120:2cm) coordinate (C1);
  \path (A2) ++(120:1cm) coordinate (C2);
   \draw[thick] (origin) -- (A3) (origin) -- (B3)  
        (A1) -- (C1) (A2) -- (C2)  
          (B2) -- (C1) (B1) -- (C2) (B3) -- (C1) 
        (B1) -- (A1) (B2) -- (A2) (B3) -- (A3);
  \draw (origin) ++(0,0.6) node {\small$\mathbf{1}$};
  \draw (origin) ++(0,1.1) node {\small$\mathbf{2}$};
  \draw (A1) ++(0,0.6) node {\small$\mathbf{3}$};
  \draw (A1) ++(0,1.1) node {\small$\mathbf{4}$};
  \draw (B1) ++(0,0.6) node {\small$\mathbf{3}'$};
  \draw (B1) ++(0,1.1) node {\small$\mathbf{4}'$};
  \draw (A2) ++(0,0.6) node {\small$\mathbf{6}$};
  \draw (B2) ++(0,0.6) node {\small$\mathbf{6}'$};
  \draw (A1) ++(120:1cm) ++(0,0.6) node {\small$\mathbf{5}$};
  \end{tikzpicture}
\end{center}
\caption{The fundamental alcoves for $G={\rm SL}_3$}\label{sl3label}
\end{figure}
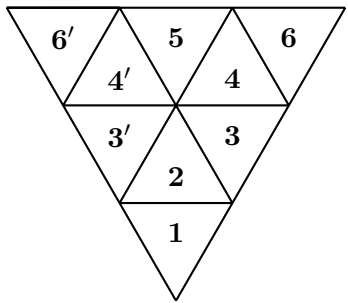
 We study the tilting modules for $\text{SL}_3(k)$ by considering the structure of the corresponding generalised Schur algebra $S(\pi)$ where $\pi=\{r \in X: r \leq 2(p-1)\rho \}$.
By Section 4.3 
 we have that $P(2)=T(5)$. Now, 2 appears in the first layer of $\Delta(2)$, second layer of $\Delta(3)$ and $\Delta(3')$, the third layer of $\Delta(4)$ and $\Delta(4')$, and fourth layer of $\Delta(5)$.  Therefore, by Corollary 2, the tilting module $T(5)$ has the following Loewy structure:
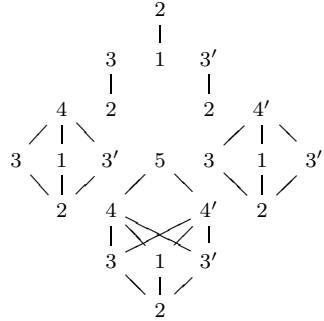
\begin{center}
\begin{minipage}{34mm}
\def\objectstyle{\scriptstyle}
\xymatrix@=6pt{
&			&		&		&2\ar@{-}[d]		&		&		&		\\
&			&		&3\ar@{-}[d]		&1				&3'\ar@{-}[d]		&		\\
&			&4\ar@{-}[d]\ar@{-}[dl]\ar@{-}[dr]	&2		&		&2		&4'\ar@{-}[d]\ar@{-}[dr]\ar@{-}[dl]		  		\\
&3			&1		&3'		&5\ar@{-}[dr]\ar@{-}[dl]		&3		&1		&3'		\\
&			&2\ar@{-}[u]\ar@{-}[ur]\ar@{-}[ul]		&4\ar@{-}[d]\ar@{-}[dr]\ar@{-}[drr]		&		&4'	\ar@{-}[d]\ar@{-}[dl]\ar@{-}[dll]	&2\ar@{-}[u]\ar@{-}[ur]\ar@{-}[ul]		&		\\
&			&		&3		&1		&3'		&		&		\\
&			&		&		&2\ar@{-}[u]\ar@{-}[ur]\ar@{-}[ul]		&		&		&		\\
}\end{minipage}
\end{center}
note that we have highlighted where the Weyl filtration lies within the Loewy layers.  This diagram is symmetrical with regards to a horizontal reflection through the middle.  Therefore using the dual of Corollary 2 to calculate the socle structure, we see that this module is rigid for all $p\geq 2h-2$.

Andersen and Kaneda show that the $p^2$-restricted tilting modules sufficiently `far from the walls' are rigid for $p\geq 3h-3$ (assuming the Lusztig conjecture).  In the case of $\SL_3(k)$ these are precisely the tilting modules of highest weight $\lambda$ such that $\langle \lambda^\vee, \alpha_1+\alpha_2\rangle \leq p^2$, and $\langle \lambda^\vee, \alpha \rangle \geq p$ for at least one simple root $\alpha\in S$.

In the case of $\SL_3(k)$, we can provide a complete solution to the question of determining when $p^2$-restricted tilting modules are rigid, as follows.
The Weyl module structure for the weights `close to the walls' is calculated in \cite{DS}.  Using this information and the arguments above one can show that
 the $p^2$-restricted tilting modules for $\SL_3(k)$ are all rigid if and only if $p\geq 2h-2$.

\medskip
\noindent
{\bf Acknowledgements.}\quad We would like to thank Steve Doty for many useful discussions 
and comments on previous versions of this manuscript.

\end{document}